\documentclass[12pt]{article}%
\usepackage[intlimits]{amsmath}
\usepackage{amssymb}
\usepackage{latexsym}
\usepackage{tikz}
\usepackage{anysize}
\usepackage{subfigure}
\usepackage[T1]{fontenc}
\usepackage[sc]{mathpazo}
\usepackage{color}
\usepackage[colorlinks]{hyperref}
\usepackage{amsfonts}
\usepackage{graphicx}%
\definecolor {refcol}{RGB}{40,0,255}
\hypersetup{colorlinks=true,allcolors=refcol}
\makeatletter
\newfont{\footsc}{cmcsc10 at 8truept}
\newfont{\footbf}{cmbx10 at 8truept}
\newfont{\footrm}{cmr10 at 10truept}
\newtheorem{theorem}{Theorem}

\newtheorem{definition}{Definition}
\newtheorem{example}{Example}

\newenvironment{proof}[1][Proof]{\noindent{\textbf {#1}  }}  {\hfill$\Box$\bigskip}
\begin{document}
\title{\textbf{Computing the $A_{\alpha}-$ eigenvalues of a bug}}

\author{Oscar Rojo\thanks{Department of
Mathematics, Universidad Cat\'{o}lica del Norte, Antofagasta, Chile.}}
\date{}
\maketitle

\begin{abstract}
Let $G$ be a simple undirected graph. For $\alpha \in [0,1]$, let
\begin{equation*}
  A_{\alpha}\left(  G\right)  =\alpha D\left(  G\right)  +(1-\alpha)A\left(
G\right)  ,
\end{equation*}
where $A(G)$ is the adjacency matrix of $G$ and $D(G)$ is the diagonal matrix of the degrees of $G$.
In particular, $A_{0}(G)=A(G)$ and $A_{\frac{1}{2}}(G)=\frac{1}{2}Q(G)$ where $Q(G)$ is the signless Laplacian matrix of $G$. A bug $B_{p,q,r}$ is a graph obtained from a complete graph
$K_{p}$ by deleting an edge and attaching paths $P_{q}$ and $P_{r}$ to its
ends. In \cite{HaSt08}, Hansen and Stevanovi\'{c} proved that, among the graphs $G$ of order $n$ and diameter $d$, the largest spectral radius of $A(G)$ is attained by the bug $B_{n-d+2,\lfloor d/2\rfloor,\lceil
d/2\rceil}$. In \cite{LiLu14}, Liu and Lu proved  the same result for the spectral radius of $Q(G)$. Let $\rho_{\alpha}(G)$ be the spectral radius of $A_{\alpha}(G)$. In this note, for a bug $B$ of order $n$ and diameter $d$, it is shown that $(n-d+2)\alpha -1$ is an eigenvalue of $A_{\alpha}(B)$ with multiplicity $n-d-1$ and that the other eigenvalues, among them $\rho_{\alpha}(B)$, can be computed as the eigenvalues of a symmetric tridiagonal matrix of order $d+1$. It is also shown that $\rho_{\alpha}(B_{n-d+2,d/2,d/2})$ can be computed as the spectral radius of a symmetric tridiagonal matrix of order $\frac{d}{2}+1$ whenever $d$ is even.
\end{abstract}

\textbf{AMS classification: }\textit{ 05C50, 15A48}

\textbf{Keywords: }\textit{convex combination of matrices; signless Laplacian;
adjacency matrix; graph diameter; spectral radius; bug.}

\section{Introduction}
Let $G=(V(G),E(G))$ be a simple undirected graph on $n$ vertices with vertex set $V(G)$ and edge set $E(G)$. Let $D(G)$ be the diagonal matrix of the degrees of $G$. Let $A(G)$ be the adjacency matrix of $G.$
In \cite{Nik17}, the family of matrices $A_{\alpha}(G)$,
\[
A_{\alpha}(G)=\alpha D(G)+(1-\alpha)A(G)
\]
with $\alpha \in [0,1]$, is introduced together with a number of some basic results and several open problems.

Observe that $A_{0}\left(  G\right)  =A\left(  G\right)  $ and $A_{1/2}\left(
G\right)  =\frac{1}{2} Q\left(  G\right)  $.

Since $A_1(G)=D(G)$, from now on, we take $\alpha \in [0,1)$.

A bug $B_{p,q,r}$ is a graph obtained from a complete graph
$K_{p}$ by deleting an edge $uv$ and attaching the paths $P_q$ and $P_r$ by one of their end vertices at $u$ and $v$, respectively. Observe that $B_{p,q,r}$ is a graph of order $p+q+r-2$ and diameter $q+r$.

\begin{example}
\label{exam2}
For instance $Bug_{5,3,4}$ is the graph
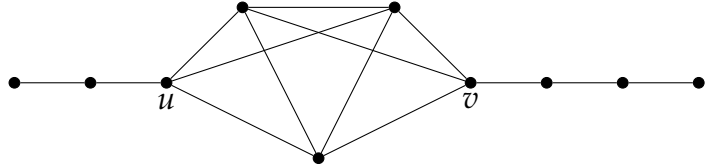
\begin{eqnarray*}
\begin{tikzpicture}
    \tikzstyle{every node}=[draw,circle,fill=black,minimum size=4pt,
                            inner sep=0pt]
                            % 0 grados y se continua con 180+(360/n)
    \draw (0,-1) node (1) [label=below:$$] {}
         (2,0) node (2) [label=below:$v$] {}
         (3,0) node (3) [label=below:$$] {}
         (4,0) node (4) [label=below:$$] {}
         (5,0) node (5) [label=below:$$] {}
         (1,1) node (6) [label=below:$$] {}
         (-1,1) node (7) [label=below:$$] {}
         (-2,0) node (8) [label=below:$u$] {}
         (-3,0) node (9) [label=above:$$] {}
         (-4,0) node (10) [label=below:$$] {};
\draw (1)--(2); \draw (2)--(3); \draw (3)--(4); \draw (4)--(5); \draw (1)--(6); \draw (1)--(7); \draw (2)--(6); \draw (2)--(7); \draw (1)--(8); \draw (8)--(9);\draw (9)--(10);\draw (8)--(7);\draw (8)--(6); \draw (7)--(6);
\end{tikzpicture}
\end{eqnarray*}
of $10$ vertices and diameter $7$ which is obtained from $K_5$ deleting the edge $uv$ and attaching the paths $P_3$ and $P_4$ at $u$ and $v$ respectively.
\end{example}

Let $\rho(M)$ be the spectral radius of the matrix $M$.

In \cite{HaSt08}, Hansen and Stevanovi\'{c} proved the following result:
\begin{theorem}
\label{ad}Let $G$ be a graph of
order $n$ and diameter $d$. If $d=1$, then
$\rho(A(G))=\rho (A(K_n))$. If $d\geq2$, then
\[
\rho(A(G)) \leq\rho(A(B_{n-d+2,\lfloor d/2\rfloor,\lceil
d/2\rceil})).
\]
The equality holds if and only if $G=B_{n-d+2,\lfloor d/2\rfloor,\lceil
d/2\rceil}$.
\end{theorem}

In \cite{LiLu14}, Liu and Lu proved  the same result for the spectral radius of $Q(G)$:
\begin{theorem}
Let $G$ be a graph of order $n$ and diameter $d$. If $d=1$, then $\rho_(Q(G)) =\rho(Q(K_{n}))$. If $d\geq2$, then
\[
\rho(Q(G)) \leq\rho(Q(B_{n-d+2,\lfloor d/2\rfloor,\lceil
k/2\rceil})).
\]
The equality holds if and only if $G=B_{n-d+2,\lfloor d/2\rfloor,\lceil
d/2\rceil}$.
\end{theorem}

Let $\rho_{\alpha}(G)$ be the spectral radius of $A_{\alpha}(G)$. From the Perron - Frobenius Theory for nonnegative matrices, for a connected graph $G$, $\rho_{\alpha}(G)$ is a simple eigenvalue of $A_{\alpha}(G)$ having a positive eigenvector.

In this note, for a bug $B$ of order $n$ and diameter $d$, we prove that $(n-d+2)\alpha -1$ is an eigenvalue of $A_{\alpha}(B)$ with multiplicity $n-d-1$ and that the other eigenvalues, among them $\rho_{\alpha}(B)$, can be computed as the eigenvalues of a symmetric tridiagonal matrix of order $d+1$. We also prove that $\rho_{\alpha}(B_{n-d+2,d/2,d/2})$ can be computed as the spectral radius of a symmetric tridiagonal matrix of order $\frac{d}{2}+1$ whenever $d$ is even. This note extends the results obtained in \cite{NEO} for the signless Laplacian matrix of bugs of order $n$ and diameter $d$.

\section{On the $\alpha-$ spectrum of a bug}
%Given two vertex disjoint graphs $G_{1}$ and $G_{2},$ the join of $G_{1}$
%and $G_{2}$ is the graph $G=G_{1}\vee G_{2}$ such that
%\begin{equation*}
%  V\left( G\right)=V\left( G_{1}\right) \cup V\left( G_{2}\right)
%\end{equation*}
% and
% \begin{equation*}
% E\left( G\right) =E(G_{1})\cup E(G_{2})\cup \left\{ xy:x\in V\left(
%G_{1}\right) ,y\in V\left( G_{2}\right) \right\}.
% \end{equation*}
We recall the notion of the $H-join$ of graphs \cite{Cardoso1, Cardoso}. Let $H$ be a graph of order $k$ with $V(H)=\{1,\ldots ,k\}$. Let $\left\{
G_{1},\ldots ,G_{k}\right\} $ be a set of pairwise vertex disjoint graphs. For $1 \leq j \leq k$, the vertex $j\in V(H)$ is assigned to the graph $G_{j}$. The $H-join$ of the graphs $G_1,\ldots,G_k$, denoted by
\begin{equation*}
G=\bigvee_{H}{\{G_{j}:1\leq j\leq k\}},
\end{equation*}
is the graph $G$ obtained from the graphs $G_{1},\ldots ,G_{k}$ and the edges
connecting each vertex of $G_{i}$ with all the vertices of $G_{j}$ if and only if $ij\in E\left( H\right) .$ That is, $G$ is the graph with vertex set
\begin{equation*}
V(G)=\bigcup_{i=1}^{k}{V(G_{i})}
\end{equation*}
 and edge set
 \begin{equation*}
E(G)=\left( \bigcup_{i=1}^{k}{E(G_{i})}\right) \cup \left( \bigcup_{ij\in
E\left( H\right) }{\{uv:u\in V(G_{i}),v\in V(G_{j})\}}\right) .
 \end{equation*}

Clearly if each $G_{i}$ is a graph of order of $n_{i}$, then $H-join$ of $G_{1},\ldots
,G_{k}$ is a graph of order $n_{1}+n_{2}+\ldots +n_{k}.$

In particular, for $i=1,\ldots,d-1$, $B_{n-d+2,i,d-i}$ is the $P_{d+1}-join$ of the regular graphs $G_{1}=\ldots
=G_{i}=K_{1},G_{i+1}=K_{n-d},G_{i+2}=\ldots =G_{d+1}=K_{1}.$ Since $B_{n-d+2,i,d-i}$ and $B_{n-d+2,d-i,i},$ are isomorphic graphs, we may take $1\leq i \leq \lfloor \frac{d}{2} \rfloor$.

\begin{example} Below are displayed the non-isomorphic bugs of order $10$ and diameter $7$:\\

$B_{5,1,6}$ is the $P_8-join$ of $G_1=K_1$, $G_2=K_3$ and $G_i=K_1$ for $i=3,\ldots,8$:
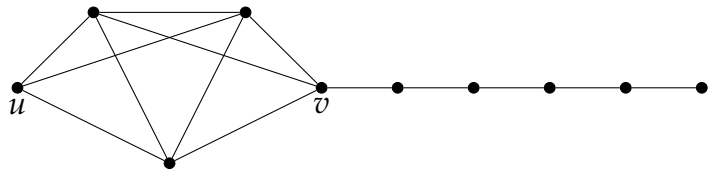
\begin{eqnarray*}
\begin{tikzpicture}
    \tikzstyle{every node}=[draw,circle,fill=black,minimum size=4pt,
                            inner sep=0pt]
                            % 0 grados y se continua con 180+(360/n)
    \draw (0,-1) node (1) [label=below:$$] {}
         (2,0) node (2) [label=below:$v$] {}
         (3,0) node (3) [label=below:$$] {}
         (4,0) node (4) [label=below:$$] {}
         (5,0) node (5) [label=below:$$] {}
         (1,1) node (6) [label=below:$$] {}
         (-1,1) node (7) [label=below:$$] {}
         (-2,0) node (8) [label=below:$u$] {}
         (7,0) node (9) [label=above:$$] {}
         (6,0) node (10) [label=below:$$] {};

\draw (1)--(2); \draw (2)--(3); \draw (3)--(4); \draw (4)--(5); \draw (1)--(6); \draw (1)--(7); \draw (2)--(6); \draw (2)--(7); \draw (1)--(8); \draw (10)--(9);\draw (5)--(10);\draw (8)--(7);\draw (8)--(6); \draw (7)--(6);
\end{tikzpicture}
\end{eqnarray*}

$B_{5,2,5}$ is the $P_8-join$ of $G_1=G_2=K_1$, $G_3=K_3$ and $G_i=K_1$ for $i=4,\ldots,8$:
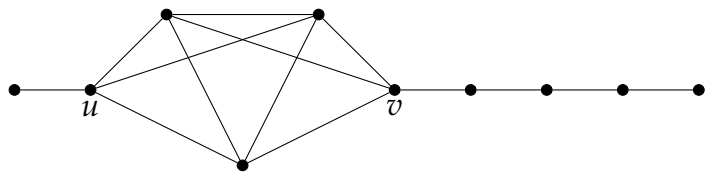
\begin{eqnarray*}
\begin{tikzpicture}
    \tikzstyle{every node}=[draw,circle,fill=black,minimum size=4pt,
                            inner sep=0pt]
                            % 0 grados y se continua con 180+(360/n)
    \draw (0,-1) node (1) [label=below:$$] {}
         (2,0) node (2) [label=below:$v$] {}
         (3,0) node (3) [label=below:$$] {}
         (4,0) node (4) [label=below:$$] {}
         (5,0) node (5) [label=below:$$] {}
         (1,1) node (6) [label=below:$$] {}
         (-1,1) node (7) [label=below:$$] {}
         (-2,0) node (8) [label=below:$u$] {}
         (-3,0) node (9) [label=above:$$] {}
         (6,0) node (10) [label=below:$$] {};

\draw (1)--(2); \draw (2)--(3); \draw (3)--(4); \draw (4)--(5); \draw (1)--(6); \draw (1)--(7); \draw (2)--(6); \draw (2)--(7); \draw (1)--(8); \draw (8)--(9);\draw (5)--(10);\draw (8)--(7);\draw (8)--(6); \draw (7)--(6);
\end{tikzpicture}
\end{eqnarray*}

$B_{5,3,4}$ is the $P_8-join$ of $G_1=G_2=G_3=K_1$, $G_4=K_3$ and $G_i=K_1$ for $i=5,\ldots,8$:
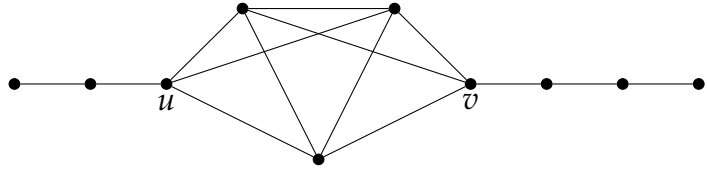
\begin{eqnarray*}
\begin{tikzpicture}
    \tikzstyle{every node}=[draw,circle,fill=black,minimum size=4pt,
                            inner sep=0pt]
                            % 0 grados y se continua con 180+(360/n)
   \draw (0,-1) node (1) [label=below:$$] {}
         (2,0) node (2) [label=below:$v$] {}
         (3,0) node (3) [label=below:$$] {}
         (4,0) node (4) [label=below:$$] {}
         (5,0) node (5) [label=below:$$] {}
         (1,1) node (6) [label=below:$$] {}
         (-1,1) node (7) [label=below:$$] {}
         (-2,0) node (8) [label=below:$u$] {}
         (-3,0) node (9) [label=above:$$] {}
         (-4,0) node (10) [label=below:$$] {};

\draw (1)--(2); \draw (2)--(3); \draw (3)--(4); \draw (4)--(5); \draw (1)--(6); \draw (1)--(7); \draw (2)--(6); \draw (2)--(7); \draw (1)--(8); \draw (8)--(9);\draw (9)--(10);\draw (8)--(7);\draw (8)--(6); \draw (7)--(6);
\end{tikzpicture}
\end{eqnarray*}
\end{example}

It is immediate that if $G$ is a $r-$ regular graph of order $n$ then $A_{\alpha}(G)=\alpha r I_n + (1-\alpha) A(G)$ and $\rho(A(G))=r$. Hence $\rho_{\alpha}(G)=\alpha r + (1-\alpha) r= r$ for any $r$-regular graph $G$.

Let $\sigma(M)$ be the spectrum of a matrix $M$. 
In \cite{Cardoso1}, Theorem $5$, the spectrum of the adjacency matrix of the $H$- join of regular graphs is obtained. The version of the corresponding result for the spectrum of $A_{\alpha}$ is given below and its proof is similar.

\begin{theorem}
\label{hjoin}
Let $H$ be a graph of order $k$. Let $G=\bigvee_{H}{\{G_{j}:1\leq j\leq k\}}$. If each $G_j$ is a $r_{j}$-regular
graph of order $n_j$ then
\begin{equation*}
  \sigma(A_{\alpha}(G)) =\cup _{G_{j}\neq K_{1}}\left\{
\alpha s_{j}+\lambda :\lambda \in \sigma \left(A_{\alpha}(G_{j}) \right)
\backslash \left\{ r_{j}\right\} \right\} \cup \sigma \left( M(G) \right)
\end{equation*}
where $M(G)$ is a matrix of order $k \times k$ given by
\begin{equation}\label{MG}
M(G)=\left[
\begin{array}{cccc}
\alpha s_{1}+r_{1} & \beta \delta _{12}  \sqrt{n_{1}n_{2}} & \ldots  & \beta \delta _{1k}\sqrt{%
n_{1}n_{k}} \\
\beta \delta _{12}\sqrt{n_{1}n_{2}} & \alpha s_{2}+r_{2} & \ddots  & \vdots  \\
\vdots  & \ddots  & \ddots  & \beta \delta _{\left( k-1\right) k}\sqrt{n_{k-1}n_{k}%
} \\
\beta \delta _{1k}\sqrt{n_{1}n_{k}} & \ldots  &\beta  \delta _{\left( k-1\right) k}\sqrt{%
n_{k-1}n_{k}} & \alpha s_{k}+r_{k}%
\end{array}%
\right]
\end{equation}
with
\[
\beta=1-\alpha,
\]
\[
\delta _{ij}=\left\{
\begin{array}{c}
1\text{ if }ij\in E\left( H\right)  \\
0\text{ otherwise}%
\end{array}%
\right.
\]
and, for $j=1,2,\ldots ,k,$
\begin{equation*}
  s_{j}=\sum_{jl\in E\left( H\right) \ }n_{l}.
\end{equation*}
\end{theorem}

For brevity, let $B(i)= B_{n-d+2,i,d-i}$ with $1\leq i\leq \lfloor \frac{d}{2}\rfloor$.

Since, for each $1\leq i\leq k$, $B(i)$ is the $P_{d+1}-join$ of the regular graphs $G_{1}=\ldots
=G_{i}=K_{1},G_{i+1}=K_{n-d},G_{i+2}=\ldots =G_{d+1}=K_{1}$, Theorem \ref{hjoin} can be applied to determine the spectrum of $A_{\alpha}(B(i))$ with the advantage that the matrix $M(B(i))$ in (\ref{MG}), of order $d+1$,  is a symmetric tridiagonal matrix. For instance, the matrices $M(B(1))$, $M(B(2))$ and $M(B(3))$ are
\begin{equation*}
  M(B(1))=\left[
        \begin{array}{cccccccc}
          \alpha(n-d) & \beta \sqrt{n-d} &  &  &  &  & & \\
         \beta \sqrt{n-d}  & 2\alpha +n-d-1 & \beta \sqrt{n-d} &  &  &  & & \\
           & \beta \sqrt{n-d} & \alpha(n-d+1) & \beta &  &  & & \\
           &  & \beta & 2\alpha & \ddots &  & & \\
           &  &  & \ddots & \ddots & \ddots &  &\\
           &  &  &  & \ddots & 2\alpha & \beta &\\
           &  &  &  &  & \beta & 2\alpha & \beta \\
           &  &  &  &  &  &  \beta & \alpha
        \end{array}
      \right],
\end{equation*}
\begin{equation*}
M(B(2))=\left[
        \begin{array}{cccccccc}
          \alpha & \beta &  &  &  &  & & \\
         \beta & \alpha( n-d+1) & \beta \sqrt{n-d} &  &  &  & & \\
           & \beta \sqrt{n-d} & 2\alpha + n-d-1 & \beta\sqrt{n-d} &  &  & & \\
           &  & \beta \sqrt{n-d} & \alpha( n-d+1) & \beta &  & & \\
           &  &  & \beta & 2\alpha & \ddots &  &\\
           &  &  &  & \ddots & \ddots & \beta &\\
           &  &  &  &  & \beta & 2\alpha & \beta \\
           &  &  &  &  &  &  \beta & \alpha
        \end{array}
      \right]
 \end{equation*}
and
\begin{equation*}
M(B(3))=\left[
        \begin{array}{ccccccccc}
          \alpha & \beta &  &  &  &  & & &\\
         \beta & 2 \alpha & \beta &  &  &  & & &\\
           & \beta & \alpha(n-d+1) & \beta\sqrt{n-d} &  &  & & & \\
           &  & \beta \sqrt{n-d} & 2\alpha + n-d-1 & \beta &  & & &\\
           &  &  & \beta\sqrt{n-d} & \alpha(n-d+1) & \ddots &  & &\\
           &  &  &  & \beta & 2\alpha & \ddots & &\\
           &  &  &  &  & \ddots & \ddots & \beta & \\
           &  &  &  &  &  &  \ddots & 2\alpha & \beta\\
           &  &  &  &  &  &  & \beta & \alpha
        \end{array}
      \right].
 \end{equation*}
%The diagonal entries (\emph{d.e.}) and the codiagonal entries (\emph{cd.e.}) of $B_(1),B_(2)$ and $B_(3)$ are displayed below
%\begin{equation*}
%B(1):
%  \begin{array}{ccccccccc}
%    d. e. & \alpha (n-d) & 2 \alpha + n-d-1 & \alpha (n-d+1) & 2\alpha & \ldots & \ldots & 2\alpha & \alpha \\
%    cd. e. & \beta \sqrt{n-d} & \beta \sqrt{n-d} & \beta & \ldots & \ldots & \ldots & \beta & \\
%  \end{array}
%  \end{equation*}
% \begin{equation*}
%B(2):
%  \begin{array}{cccccccccc}
%    d. e. & \alpha & \alpha( n-d+1) & 2\alpha + n-d-1 & \alpha (n-d+1) & 2 \alpha & \ldots & 2 \alpha& \alpha \\
%    cd. e. & \beta & \beta \sqrt{n-d} & \beta \sqrt{n-d} & \beta & \ldots & \ldots & \beta & \\
%  \end{array}
%  \end{equation*}
% \begin{equation*}
%B(3):
%  \begin{array}{ccccccccccc}
%    d. e. & \alpha & 2\alpha & \alpha(n-d+1) & 2\alpha+n-d-1 & \alpha(n-d+1) & 2\alpha & \ldots & 2\alpha & \alpha\\
%    cd. e. & \beta & \beta & \beta \sqrt{n-d} & \beta \sqrt{n-d} & \beta & \ldots & \ldots &\beta \\
%  \end{array}
%  \end{equation*}
\begin{definition} Let $\alpha \in [0,1)$ and $\beta =1-\alpha$. Let
\[
R =\left[
\begin{array}{ccc}
\alpha (n-d +1) & \beta \sqrt{n-d} &  0\\
\beta \sqrt{n-d} & 2\alpha +(n-d-1) & \beta \sqrt{n-d} \\
0 & \beta \sqrt{n-d} &\alpha (n-d +1)%
\end{array}%
\right],
\]

\[
T_{1}=[\alpha]
\]
 and, for $s\geq 2$, let
\begin{equation*}
T_{s}=\left[
\begin{array}{ccccc}
\alpha & \beta &  &  &  \\
\beta & 2 \alpha & \beta &  &  \\
& \beta & \ddots  & \ddots  &  \\
&  & \ddots  & 2 \alpha & \beta \\
&  &  & \beta & 2 \alpha%
\end{array}%
\right]
\end{equation*}
of order $s \times s$.
\end{definition}

 Let $I$ be the identity matrix, $0$ the zero matrix, $J$ the exchange matrix (the matrix with ones in the secondary diagonal and zeros elsewhere) and $F$ the matrix whose entries are zeros except for the entry in the last row and first column which is equal to $1$. All of them of the appropriate order.

 The use of Theorem \ref{hjoin} yields to :

\begin{theorem}
\label{bug}
Let $\alpha \in [0,1)$ and $\beta =1-\alpha$. The eigenvalues of $A_{\alpha}(B(i))$ are $(n-d+2)\alpha-1$ with multiplicity $n-d-1$ and the other eigenvalues, among them $\rho_{\alpha}(B(i))$, can be computed as the eigenvalues of the $\left(
d+1\right) \times \left( d+1\right) $ symmetric tridiagonal matrix
\begin{equation*}
  M(B(i))=\left[
\begin{array}{cc}
X_{i} & \beta F \\
\beta F^{T} & JT_{d-i-1}J%
\end{array}%
\right]
\end{equation*}
where
\begin{equation*}
  X_{1}=\left[
\begin{array}{ccc}
\alpha (n-d) & \beta \sqrt{n-d} & 0  \\
\beta \sqrt{n-d} & 2 \alpha +n-d-1 &\beta \sqrt{n-d} \\
0 & \beta\sqrt{n-d} & \alpha (n-d+1)%
\end{array}%
\right]
\end{equation*}
whenever $i=1$, and
\begin{equation*}
  X_{i}=\left[
\begin{array}{cc}
T_{i-1} & \beta F \\
\beta F^T & R
\end{array}%
\right]
\end{equation*}
whenever $2\leq i\leq \left\lfloor \frac{d}{2}\right\rfloor $.
\end{theorem}

Since each $M(B(i))$ is a tridiagonal matrix with nonzero codiagonal entries, its eigenvalues are simple.
\begin{example}
Consider the bug $B_{8,2,3}$
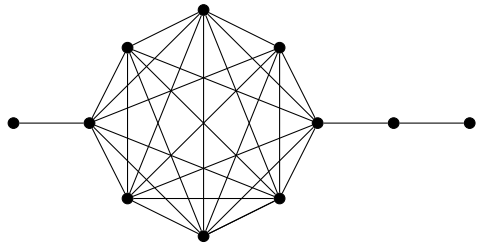
\begin{eqnarray*}
\begin{tikzpicture}
    \tikzstyle{every node}=[draw,circle,fill=black,minimum size=4pt,
                            inner sep=0pt]
                            % 0 grados y se continua con 180+(360/n)
   \draw (-2.5,0) node (1) [label=below:$$] {}
         (-1.5,0) node (2) [label=below:$$] {}
         (-1,-1) node (3) [label=below:$$] {}
         (0,-1.5) node (4) [label=below:$$] {}
         (1.5,0) node (6) [label=below:$$] {}
         (1,-1) node (5) [label=below:$$] {}
         (1,1) node (7) [label=below:$$] {}
         (0,1.5) node (8) [label=below:$$] {}
         (-1,1) node (9) [label=above:$$] {}
         (2.5,0) node (10) [label=below:$$] {}
         (3.5,0) node (11) [label=below:$$] {};

\draw (1)--(2); \draw (2)--(3); \draw (2)--(4); \draw (2)--(5); \draw (3)--(6); \draw (2)--(7); \draw (2)--(8); \draw (2)--(9); \draw (3)--(4); \draw (3)--(5);\draw (3)--(7);\draw (8)--(3);\draw (3)--(9); \draw (4)--(5);\draw (4)--(6); \draw (7)--(4); \draw (4)--(8);\draw (4)--(9);\draw (5)--(6);\draw (5)--(7); \draw (5)--(8); \draw (5)--(9); \draw (6)--(7); \draw (6)--(8);\draw (6)--(9);\draw (6)--(10);\draw (10)--(11); \draw (4)--(5); \draw (7)--(8); \draw (8)--(9);
\end{tikzpicture}
\end{eqnarray*}
This bug is the $P_6-join$ of $G_1=G_2=K_1$, $G_3=K_6$ and $G_i=K_1$ for $i=4,5,6$. From Theorem \ref{bug}, to four decimal places and for $\alpha=0.6$, the eigenvalues of this bug are $(11-5+2)\cdot 0.6 -1=3.8$ with multiplicity $11-5-1=5$ and the eigenvalues of the $ 6 \times 6$ matrix
\begin{equation*}
M(B(2))=M(B_{8,2,3})=\left[
  \begin{array}{cccccc}
    0.6 & 0.4 &  &  & & \\
    0.4 & 4.2 & 0.9798 & &  &\\
     & 0.9798 & 6.2 & 0.9798 & & \\
     &  & 0.9798 & 4.2 & 0.4  & \\
     &  &  & 0.4 & 1.2 & 0.4\\
     & & & &0.4 & 0.6\\
  \end{array}
\right].
\end{equation*}
The eigenvalues of $M(B_{8,2,3})$ are $0.3909, 0.5539, 1.3521, 3.5403, 4.2486, 6.9144$.
\end{example}

%\begin{theorem}
%\label{comp}
%Let $G\in \mathcal{G}_{n,d}.$
%
%$\left( a\right) $ If $d=3$ then the largest $q_{1}\left( G\right) $ can be
%computed as the largest eigenvalue of the symmetric tridiagonal
%matrix $M(B(1))$ of order $4$ with diagonal entries
%\begin{equation*}
%  n-3, 2(n-3),n-2,1
%\end{equation*}
%and codiagonal entries
%\begin{equation*}
%  \sqrt{n-3},\ \sqrt{n-3},1
%\end{equation*}
%
%
%$\left( b\right) \ $If $d\geq 4$ then the largest $q_{1}\left( G\right) $
%can be computed as the largest eigenvalue of the symmetric tridiagonal matrix $M(B(\lfloor \frac{d}{2} \rfloor))$ of order $d+1$ with diagonal entries
%\[
%1,\overset{\left\lfloor \frac{d}{2}\right\rfloor -2}{\overbrace{2,\ldots ,2}}%
%,n-d +1,2(n-d) ,n-d+1\overset{\left\lceil \frac{d}{2}\right\rceil -2}%
%{,\overbrace{2,\ldots ,2}},1
%\]
%and codiagonal entries%
%\[
%\overset{\left\lfloor \frac{d}{2}\right\rfloor -1}{\overbrace{1,\ldots ,1}},%
%\sqrt{n-d},\sqrt{n-d},\overset{\left\lceil \frac{d}{2}\right\rceil -1%
%}{\overbrace{1,\ldots ,1}}.
%\]
%\end{theorem}

%We now prove that $q_1(B(\lfloor \frac{d}{2} \rfloor)$ can be computed as the largest eigenvalue of a symmetric tridiagonal matrix of order $\frac{d}{2}+1$ whenever $d$ is an even integer.

\begin{theorem}
Let $\alpha \in [0,1)$ and $\beta=1-\alpha$. Let $d\geq 4$ be an even integer. Then $\rho_{\alpha}(B_{n-d+2,\frac{d}{2},\frac{d}{2}})$ can be computed as the largest eigenvalue of the symmetric tridiagonal matrix
\begin{equation*}
  \left[
    \begin{array}{cccccc}
      \alpha & \beta &  &  &  &  \\
      \beta & 2\alpha & \ddots &  &  &  \\
       & \ddots & \ddots & \ddots & &  \\
       &  & \ddots & 2\alpha & \beta &  \\
       &  &  & \beta & (n-d+1)\alpha & \beta \sqrt{2(n-d}) \\
       &  &  &  & \beta \sqrt{2(n-d}) & 2\alpha+n-d-1 \\
    \end{array}
  \right]
\end{equation*}
 of order $\frac{d}{2}+1$.
\end{theorem}

\begin{proof}
From Theorem \ref{bug}, the eigenvalues of $B_{n-d+2,\frac{d}{2},\frac{d}{2}}$ are $(n-d+2)\alpha+1$ with multiplicity $n-d-1$ and the eigenvalues of the matrix
\[
M(B(\frac{d}{2}))=\left[
\begin{array}{ccc}
S & \mathbf{b} & 0 \\
\mathbf{b}^{T} & 2\alpha+n-d+1  & \mathbf{b}^{T}J \\
0 & J\mathbf{b} & JSJ%
\end{array}%
\right]
\]%
of order $d+1$ where%
\[
S=\left[
\begin{array}{ccccc}
\alpha & \beta &  &  &  \\
\beta & 2\alpha & \ddots  &  &  \\
& \ddots  & \ddots  & \ddots &  \\
&  & \ddots & 2\alpha & \beta \\
&  &  & \beta & (n-d+1)\alpha%
\end{array}%
\right]
\]%
of order $\frac{d}{2},$ $\mathbf{b}^{T}=\left[
\begin{array}{ccccc}
0 & \cdots  & \cdots  & 0 & \beta \sqrt{n-d}%
\end{array}%
\right] ,\ J$ is the reverse matrix and $0$ is the zero matrix, all of them of the
appropriate sizes. Consider the orthogonal matrix%
\[
Q=\frac{1}{\sqrt{2}}\left[
\begin{array}{ccc}
I & \mathbf{0} & J \\
\mathbf{0}^{T} & \sqrt{2} & \mathbf{0}^{T} \\
-J & \mathbf{0} & I%
\end{array}%
\right] .
\]%
An easy calculation shows that%
\[
QM(B(\frac{d}{2}))Q^{T}=\left[
\begin{array}{ccc}
S & \sqrt{2}\mathbf{b} & 0 \\
\sqrt{2}\mathbf{b}^{T} & 2\alpha +n-d-1 & \mathbf{0}^{T} \\
0 & \mathbf{0} & JSJ%
\end{array}%
\right] .
\]%
Then the eigenvalues of $M(B(\frac{d}{2}))$ are the eigenvalues of $\left[
\begin{array}{cc}
S & \sqrt{2}\mathbf{b} \\
\sqrt{2}\mathbf{b}^{T} & 2\alpha+n-d-1
\end{array}%
\right] $ and the eigenvalues of $S.$ Since the eigenvalues of $S$ strictly interlace
the eigenvalues of \\ $\left[
\begin{array}{cc}
S & \sqrt{2}\mathbf{b} \\
\sqrt{2}\mathbf{b}^{T} & 2\alpha+n-d-1
\end{array}%
\right] ,$ the result follows.
\end{proof}


\begin{thebibliography}{99}%                                                                                            %

\bibitem{NEO} N. Abreu, E. Lenes, and O. Rojo, Computing the maximal signless Laplacian index among graphs of prescribed order and diameter, \emph{Proyecciones Journal of Mathematics} \textbf{34} (2015), 379-390.
    
    \bibitem{Cardoso1} D. M. Cardoso, M. A. A. de Freitas, E. Martins., M. Robbiano, Spectra of graphs obtained by a generalization of the join graph operation, \emph{Discrete Mathematics} \textbf{313} (2013), 733-741.
\bibitem{Cardoso} D. M. Cardoso, E. Martins., M. Robbiano, O. Rojo, Eigenvalues of a H-generalized operation constrained by vertex subsets, \emph {Linear Algebra Appl.} \textbf{438} (2013), 3278 - 3290.

\bibitem {HaSt08}P. Hansen and D. Stevanovi\'{c}, On bags and bugs,
\emph{Discrete Appl. Math.} \textbf{156} (2008), 986--997.

\bibitem {LiLu14}H. Liu and M. Lu, A conjecture on the diameter and signless
Laplacian index of graphs, \emph{Linear Algebra Appl.} \textbf{450} (2014), 158--174.

%\bibitem {LLT04}H. Liu, M. Lu, and F. Tian, On the spectral radius of graphs
%with cut edges, \emph{Linear Algebra Appl. }\textbf{389} (2004), 139--145.

\bibitem {Nik17}V. Nikiforov, Merging the $A$- and $Q$-spectral theories,
\emph{Applicable Analysis and Discrete Math.} $\mathbf{11}$ (2017), 81--107.

%\bibitem {Ste14}D. Stevanovi\'{c}, Spectral radius of graphs, \emph{Academic
%Press, }2014, pp.166.
%
%\bibitem {StHa08}D. Stevanovi\'{c} and P. Hansen, The minimum spectral radius
%of graphs with a given clique number, \emph{Electron. J. Linear Algebra}
%\textbf{17} (2008), 110--117.
%
%\bibitem {WXH05}B.F. Wu, E.L. Xiao, and Y. Hong, The spectral radius of trees
%on $k$ pendant vertices, \emph{Linear Algebra Appl. }\textbf{395} (2005), 343--349.
%
%\bibitem {ZHG14}J.-M. Zhang, T.-Z. Huang, and J.-M. Guo, The smallest spectral
%radius of graphs with a given clique number, \emph{The Scientific World
%Journal }2014, Article ID 232153
%
%

\end{thebibliography}
\end{document}